\documentclass{article}
\usepackage{latexsym}
\usepackage{amsmath,enumerate}
\usepackage{amsthm}
\usepackage {ae}
\usepackage {color}
\usepackage {graphicx}
\usepackage[english]{babel}
\usepackage {graphics}
\usepackage {makeidx}
\usepackage [T1]{fontenc}
\usepackage {amsmath}
\usepackage{amssymb}
\usepackage {hyperref}
\usepackage{amsfonts}
\usepackage{float}
\usepackage{subfigure}
\usepackage[mathscr]{eucal}
\newtheorem{Proposition}{Proposition}
\newtheorem{Lemma}{Lemma}

\def\exp{\mbox{{\rm exp}}}

\def\#{{{\cal D}_h}}

\def\va{\raise 2pt\hbox{,}}

\title{\bf An asymptotic preserving scheme for kinetic models
    for chemotaxis phenomena}

\author{\textbf{A. Bellouquid}$^a$ and \textbf{J. Tagoudjeu}$^b$ \\
        $^a$ Cadi Ayyad University, ENSA Marrakech, Morocco\\
        a.bellouquid@uca.ma \\
        $^b$ University of Yaound\'{e} I, ENSP and CETIC, Cameroon\\
        jtagoudjeu@gmail.com, jacques.tagoudjeu@polytechnique.cm}
\date{}

\begin{document}

\maketitle

\begin{abstract} In this paper, we propose a numerical scheme to solve the
kinetic model for chemotaxis phenomena. Formally, this scheme is
shown to be uniformly stable with respect to the small parameter,
consistent with the fluid-diffusion limit (Keller-Segel model).
Our approach is based on the micro-macro decomposition which leads
to an equivalent formulation of the kinetic model that couples a
kinetic equation with macroscopic ones. This method is validated
with various test cases and compared to other standard methods.
\end{abstract}
\textbf{keyword}:
Asymptotic preserving scheme; Kinetic theory; Micro-macro
decomposition; Chemotaxis phenomena


\section{Introduction}

Chemotaxis is a process by which cells change their movement by reacting to the presence of a chemical substance. Cells approach to chemically favorable environments and avoid unfavorable ones. In a simple description, where we only consider cells and a chemical substance (the chemo-attractant), a model for the space and time evolution of the density $n = n(t, x)$ of cells and the chemical concentration $S = S(t, x)$ at time $t$ and position $x$  has been introduced by Patlak \cite{Pat} and Keller-Segel \cite{Kel} reads:
\begin{equation}\label{ma} \left\{
\begin{array}{l}
\partial_t n + \nabla_{x} \cdot \left( n\, \chi (S)  \nabla_{x} {S}  -D_n  \nabla_x n \right)=0,  \\
\partial_t S -D_S\Delta S = H(n,S),
\end{array} \right.
\end{equation}
where $\nabla_x$ denotes the gradient with respect to the spatial variable, while the positive defined constants $ D_S$  and $D_n$ are the diffusivity of the chemo-attractant and of the cells, respectively, and $\chi$ is the chemotactic sensitivity.

In general the substance $S$ does not only diffuse in the substrate, but it can also be produced by the bacteria themselves. The role of t $H(n,S)$ consists in modeling the interaction between both quantities. A typical example is  $H(n,S)= a n -b S$, which describes the production of the chemo-attractant by the bacteria at a constant rate $a$ as well as chemical decay with relaxation time $\frac{1}{b}$. Since the bacterial movement is directed toward the higher concentrations of $S$, the coupling is attractive. A deep insight into the phenomenological derivation of Keller and Segel types models is given in the survey \cite{[HP1]}.

The qualitative analysis of Keller-Segel models has attracted several mathematicians and a variety of interesting results have been produced. The surveys \cite{[H1]} and \cite{[BB5]} provide a detailed review and critical analysis of the qualitative properties of the solutions to problems related to the application to various biological contexts.

An alternative modeling approach has been introduced by the mesoscopic description which bridges the interaction of stochastic  particle  to macroscopic equations. This middle ground consists in describing the movement of cells by a ``run \& tumble'' process \cite{Hot1,Hot2}. Cells move along a straight line in the running phase and make reorientation as a reaction to the surrounding chemicals during the tumbling phase. This is the typical behavior that has been observed in experiments. The resulting kinetic equation, with parabolic scaling, reads
\begin{equation}\label{ki}
\left\{
\begin{array}{l}
    \varepsilon \partial_t f + v\cdot \nabla_x f = \frac{1}{\varepsilon}\mathcal{T}(S,f),\\
   \partial_t S -D_S\Delta S = H(n,S),\\
 f(0,x,v)=f_0(x,v), \end{array} \right.
\end{equation}
where $f(t,x,v)$ denotes the density of cells, depending on time $t$, position $x\in \Omega \subset \mathbb{R}^d$ and velocity $v \in V \subset \mathbb{R}^d$. $\mathcal{T}$ is an operator, which models the change of direction of cells and $\varepsilon$ is a time scale which here refers to the turning frequency. The function $S(t,x)$ is the chemical concentration, where $n$ denotes the density of cells, and is given by
$$ \displaystyle{n(t,x) = \int_V f(t,x,v) dv}.$$


Starting with the kinetic equation (\ref{ki}), one can (at least formally) derive the macroscopic limit  (\ref{ma}) as $\varepsilon \rightarrow 0$.  Various asymptotic limits, including hyperbolic limits,  have been investigated in~\cite{BB1,BB3,BDE,Perth05,HIC}.

The aim of this paper is the development of numerical schemes to solve the kinetic equation by methods that are uniformly stable along the transition from kinetic regime to the fluid regime. The main difficulty is due to the term $\frac{1}{\varepsilon}$ which becomes stiff when $\varepsilon$ is close to zero (macroscopic regime). In this case, solving the kinetic equation by a standard explicit numerical scheme requires the use of a time step of the order of $\varepsilon$, which leads to very expensive numerical computations for small $\varepsilon$. To avoid this difficulty, it is necessary to use an implicit or semi-implicit time
discretization for the collision part. In fact, such numerical schemes should also have a correct asymptotic behavior, namely for small parameter $\varepsilon$, the schemes should degenerate into a good approximation of the  asymptotics (Keller-Segel model) of the kinetic equation. This property is often called ``asymptotic preserving'', and has been introduced in \cite{J2} for numerical schemes that are stable with respect to a small parameter $\varepsilon$ and degenerate into a consistent numerical scheme for the limit model when $\varepsilon \rightarrow 0$.

Considering that this paper deals with  asymptotic preserving scheme (AP), one also has to mention that  there are different approaches to construct such schemes for kinetic models in various contexts. We mention for instance approaches based on domain decompositions, separating the macroscopic (fluid) domain from the microscopic (kinetic) one (see \cite{CR1, D1}). There are other kind of (AP) schemes for kinetic equations, which are based on the use of time relaxed techniques where the Boltzmann collision operator is discretized by a spectral or a Monte-Carlo method (see~\cite{P1,PR1,PR}). Various techniques have also been developed to design multiscale numerical methods which are based on splitting strategy \cite{Carr11, COR, JP, JRT,Kl1}, penalization procedure \cite{FJ,FJ1,JB,BS} or micro-macro decomposition which first was used by Liu and Yu for theoretical study of the fluid limit of the Boltzmann equation \cite{Li}. It
was then used to develop an AP scheme for different asymptotics (diffusion, fluid, high-field, ...), see~\cite{Lem,Benn09,CG,CCL,L6,Ji,BLM,LMH}.

In this paper, we extend a method of micro-macro decomposition in order to construct asymptotic preserving schemes (AP) for kinetic equations describing chemotaxis phenomena. Our strategy  consists in rewriting the kinetic equation as a coupled system of kinetic part and macroscopic one, by using the micro-macro decomposition of the distribution function. Indeed, this function is decomposed into its corresponding  equilibrium distribution plus the deviation. By using a classical  projection technique, we obtain an evolution equation for the macroscopic parameters of the equilibrium coupled to a kinetic equation for the non-equilibrium part. Although our approach is rather general to apply to a very large class of collision operators, the numerical tests shown in our  work were obtained with very simple model. The outline of the contents are the following. In Section~2, we present the kinetic model  and its properties. The micro-macro decomposition, the corresponding formulation of the kinetic equation, and the  macroscopic limit are presented in Section~3. Our numerical scheme is presented in Section~4. Finally, a numerical test is presented in Section~5.

\section{The kinetic model}

This section provides a description of the kinetic model in the first equation in (\ref{ki}), where the turning kernel $\mathcal{T}$ defines the probability  density of the random velocity jump of cells from $v'$  to $v$. To derive the Keller- Segel equation (\ref{ma})
as $\varepsilon \rightarrow 0$, one has to incorporate both $O(1)$ and $O(\varepsilon)$ scale into $ \mathcal{T}$.

We suppose,  as in \cite{BB1,Perth05,HO1,HO2}, the following perturbation of the turning operator:
\begin{equation}\label{eq3}
    \mathcal{T}(S,f) = \mathcal{T}_0(f)+\varepsilon \mathcal{T}_1(S)(f),
\end{equation}
where $\mathcal{T}_0$, supposed independent of $S$,  represents the dominant part of the turning kernel
modeling the tumble process in the absence of chemical substance, while $\mathcal{T}_1(S)$ defines the perturbation due to chemical cues.

Let us now state the assumptions on the turning operators $\mathcal{T}_0$ and $\mathcal{T}_1(S)$ which are necessary to develop the perturbation approach:

    \vskip.1cm \noindent $\bullet$ The operators $\mathcal{T}_0$ and $\mathcal{T}_1$ preserve the local mass:
        \begin{equation}\label{co}
            \int_V \mathcal{T}_0(f) dv = \int_V \mathcal{T}_1(S,f)dv = 0,\quad \text{for any} \quad S\geq 0.
        \end{equation}

\vskip.1cm \noindent $\bullet$ There exists a bounded velocity distribution $M(v)>0$, independent of $x$ and $t$, such that the flow produced by the equilibrium distribution $M$ vanishes, and $M$ is normalized:
        \begin{equation}\label{M}
            \int_{V} vM(v)dv =0, \qquad \int_{V} M(v)dv =1.
        \end{equation}

\vskip.1cm \noindent $\bullet$  The detailed balance
        \begin{equation}\label{B}
            T_0(v',v)M(v) = T_0(v,v')M(v')
        \end{equation}
        holds.

\vskip.1cm \noindent $\bullet$ The kernel $T_0(v,v')$  is bounded, and there exists a constant $\sigma >0$ such that
        \begin{equation}\label{T}
            T_0(v,v') \geq \sigma M, \quad \forall (v,v') \in V \times V, \quad x \in \mathbb{R}^d, \quad t>0.
        \end{equation}

The most commonly used assumption on the turning operators $\mathcal{T}_i$ , $i = 0, 1$, is that
they are both linear integral operators with respect to $f$ and read:
 \begin{equation}\label{eq4}
    \mathcal{T}_i(S,f) = \int_V
    (T_i(S,v,v')f(t,x,v')-T_i(S,v',v)f(t,x,v))dv',
 \end{equation}
where the turning kernel $T_1(S, v, v')$ describes the reorientation of cells, i.e. the random velocity changes from $v$
to $v'$ and may depend on the chemo-attractant concentration and its derivatives.

 Technical calculations (see \cite{BB1, CMB}),   namely by integration over $v$,
interchanging $v$ by $v'$, and using (\ref{B}), yields:
 \begin{equation} \label{ee}
 \resizebox{0.92\hsize}{0.5cm}{$\int_{V} \mathcal{T}_0(g) \frac{h(v)}{M(v)}\, dv = \frac{1}{2}
 \int_{V}\int_{V} \Psi[M]\left(\frac{g(v')}{M(v')} -\frac{g(v)}{M(v)}\right)\times \left(\frac{h(v)}{M(v)}-\frac{h(v')}{M(v')}\right)\, dv\,dv',$}
\end{equation}
where
$$
\Psi[M] = \frac{1}{2}\, \big(T_0(v,v') M(v') +T_0(v',v)M(v) \big).
$$

In particular Eq.(\ref{ee})  shows that the operators $\mathcal{T}_0$, is a self-adjoint and the following equality:
\begin{equation} \label{est} - \int_{V} \mathcal{T}_0(h)
\frac{h(v)}{M(v)}\, dv = \frac{1}{2} \int_{V}\int_{V} \Psi [M]
\bigg(\frac{h(v)}{M(v)}- \frac{h(v')}{M(v')} \bigg)^{2}\, dv\,dv'
\geq 0
\end{equation}
holds true.

Moreover, for $\int_{V} h(v)\,dv =0$, Eq.(\ref{est}) and the estimate (\ref{T}) yield:
\begin{equation}  \label{sig} -\int_{V}  \mathcal{T}_0(h)
\frac{h(v)}{M(v)}\, dv \geq   \sigma
\int_{V}\frac{h^2(v)}{M(v)}\,dv,
\end{equation}
which shows that  $\mathcal{T}_0$ is a Fredholm operator in the space $L^2(V, \frac{dv}{M(v)})$. Therefore,  the following result
defines the properties of the operator  $\mathcal{T}_0$:

\begin{Lemma}
\label{LE1} Suppose that  Assumptions (\ref{M})-(\ref{T})  hold. Then, the following properties of the operators $\mathcal{T}_0$ hold true:
\begin{itemize}
\item[i)] The operator $\mathcal{T}_0$ is self-adjoint in the space $\displaystyle{ {L^{2}\left(V ,\frac{dv}{M}\right)}}$.

\item[ii)] For $\displaystyle{ f \in L^{2}\left(V, \frac{dv}{M}\right)}$, the equation $\mathcal{T}_0(g) =f$, has a unique solution $\displaystyle{ g \in L^{2}\left(V, \frac{dv}{M}\right)}$, which satisfies
$  \displaystyle{\int_{V} g(v)\, dv = 0} $ $ \hbox{if and only if} $ $\displaystyle{  \int_{V} f(v)\, dv =0.} $

\item[iii)]  The equation $\mathcal{T}_0(g) =v \,  M(v)$, has a unique solution that we call $\theta(v)$. \item[iv)]  The kernel
of $\mathcal{T}_0$ is $N(\mathcal{T}_0) = vect(M(v))$.
\end{itemize}
\end{Lemma}

\section{micro-macro decomposition and macroscopic limit}

\subsection{The micro-macro decomposition}

Let $(f, S)$ be a solution of (\ref{ki}). We decompose $f$ as follows:
$$f=M(v) n + \varepsilon  g,$$
where $n$ is the density of cells. Then $\langle g\rangle = \int_V g dv = 0$, and:
\begin{equation} \label{n1}\resizebox{.9\hsize}{!}{$
    \frac{\partial (M n)}{\partial t}  + \varepsilon \frac{\partial g}{\partial t} +
    \frac{1}{\varepsilon}  v M \cdot \nabla_{x} n + v    \cdot
    \nabla_{x} g   =
    \frac{1}{\varepsilon}\mathcal{T}^1_0(g)
    +
    \frac{1}{\varepsilon}\mathcal{T}_1(S)(M(v)n) +
    \mathcal{T}_1^1(S)(g).$}
\end{equation}

Let us now  use a projection technique to separate the macroscopic and microscopic quantities $n(t,x)$ and $g(t,x,v)$. Moreover, let $P_{M}$, denote
 the orthogonal projection onto $N(\mathcal{T}_0)$. Then
$$
P_{M}(h)= \langle h\rangle M, \quad  \mbox{for any}\quad  h\in L^2(V, \frac{dv}{M(v)}),
$$
so that one has the following:
\begin{Lemma}  \label{LE2} One has the following properties for the projection $P_{M}$:
 $$(I-P_{M})(M n)= P_{M}(g)=0,$$ $$(I-P_{M})(v M \cdot \nabla_{x} n)=v M \cdot \nabla_{x} n,$$
$$(I-P_{M})(\mathcal{T}_1(S)(M(v)n)=\mathcal{T}_1(S)(M(v)n),$$
and
$$(I-P_{M})(\mathcal{T}_1(S)(g))=\mathcal{T}_1(S)(g)$$.
\end{Lemma}

\begin{proof} The first two equalities can be rapidly proved since $P_M(M)=M$, and  $\langle g\rangle=0$, as the flux produced by $M$
is zero then
$$(I-P_M)(v M \cdot \nabla_{x} n)=v M \cdot \nabla_{x} n.$$
 In addition, using (\ref{co}) yields
$P_M(\mathcal{T}_1(S)(h))=0$ for any $h\in L^2$, so that the third and the fourth equality are completed.
\end{proof}

Taking the operator $I-P_M$ into the equation (\ref{n1}) and using Lemma \ref{LE2} yields
\begin{equation}\label{er2}\resizebox{.86\hsize}{!}{$
   \varepsilon \frac{\partial g}{\partial t} + \frac{1}{\varepsilon}  v M \cdot \nabla_{x} n + (I-P_M)(v    \cdot
    \nabla_{x} g ) = \frac{1}{\varepsilon}\mathcal{T}_0(g)+ \frac{1}{\varepsilon}\mathcal{T}_1(S)(M(v)n)
    +\mathcal{T}_1(S)(g).$}
\end{equation}
Integrating (\ref{n1}) over $v$, yields
\begin{equation}\label{er3}
    \frac{\partial n}{\partial t} +   \langle v \cdot \nabla_{x} g \rangle = 0.
\end{equation}
So that the micro-macro formulation finally reads:
\begin{equation}\label{mM}\resizebox{.86\hsize}{!}{$
\left\{
\begin{array}{l}
 \varepsilon \frac{\partial g}{\partial t} + \frac{1}{\varepsilon}  v M \cdot \nabla_{x} n + (I-P_M)(v \cdot
\nabla_{x} g ) = \frac{1}{\varepsilon}\mathcal{T}_0(g)+
\frac{1}{\varepsilon}\mathcal{T}_1(S)(M(v)n)\\ \qquad\qquad\qquad
\;\;\;\;\;\;\;\;\;\;\;\;\;\;\;\;\;\;\;\;\;\;\;\;\;\;\;\;\;\;\;\;\;\;\;\;\;\;\;\;\;\;\;\; +\mathcal{T}_1(S)(g),  \\
 \frac{\partial n}{\partial t} +  \langle v \cdot \nabla_{x} g \rangle = 0,  \\
 {}\\
\frac{\partial S}{\partial t}-D_S\Delta S = H(n,S).
\end{array} \right.$}
\end{equation}
Equations (\ref{mM}) correspond to the micro-macro formulation of the kinetic  equation (\ref{ki}) to be used to design our AP
scheme. The following Proposition shows that this formulation is indeed equivalent to the kinetic equation (\ref{ki}).

\begin{Proposition}\label{Prop1} \
\begin{enumerate}
    \item  Let $(f,S)$ be a solution of (\ref{ki}) with initial data $(f_0,S_0)$. Then $(n,g,S)$, where $n=\langle f\rangle $ and $g= \frac{1}{\varepsilon}(f-M(v) n)$ is a solution to a coupled system (\ref{mM}) with the associated initial data: 
        
        $n(t=0)=n_0 =\langle f_0 \rangle$, $g(t=0)=g_0=\frac{1}{\varepsilon}(f_0-M(v) n_0)$,  and $S(t=0)= S_0.$
    \item Conversely, if $(n, g, S)$ satisfies system (\ref{mM}) with initial data $(n_0, g_0, S_0)$ such that $\langle g_0 \rangle=0$ then $f=M(v) n +\varepsilon g$ is a solution to kinetic model (\ref{ki}) with initial data $f_0= M(v) n_0+\varepsilon g_0$, and we have $ n=\langle f \rangle$, and $\langle g\rangle=0.$ \vskip0,3cm
\end{enumerate}
\end{Proposition}

\begin{proof} The proof of 1) is detailed above. For 2), consider $(n,g, S)$ solution of (\ref{mM}). We set $f= M(v)
n+\varepsilon g$ and we show that $f$ is a solution of kinetic model (\ref{ki}). From (\ref{mM}), one has
$$
   \frac{\partial f}{\partial t} - M(v) \frac{\partial n}{\partial t} + \frac{1}{\varepsilon}  v M \cdot \nabla_{x} n + v\cdot \nabla_x g
   -P_M(v\cdot \nabla_x g)
 = \frac{1}{\varepsilon^2}\mathcal{T}_0(f)+
\frac{1}{\varepsilon}\mathcal{T}_1(S)(f).
$$
Hence
$$
   \frac{\partial f}{\partial t} - M(v) \frac{\partial n}{\partial t} +   \frac{1}{\varepsilon} v\cdot \nabla_x
   f
   -M(v)\langle v\cdot \nabla_x g \rangle
 = \frac{1}{\varepsilon^2}\mathcal{T}_0(f)+
\frac{1}{\varepsilon}\mathcal{T}_1(S)(f).
$$
Therefore using (\ref{er3}), one obtains (\ref{ki}). The property $\langle g \rangle =0$ is obtained by integrating
(\ref{er2}) over $v$, using (\ref{co}) and the property of the initial data. This completes the proof.
\end{proof}

\subsection{ The macroscopic limit}
In this subsection, the formal derivation of the macroscopic model is developed starting from the meso-macro model (\ref{ki}). The
macroscopic model has been derived mathematically in \cite{CMB}. We will see that the formal derivation is really
straightforward starting from (\ref{mM}) (compared to the equivalent formulation of (\ref{ki})), since the micro-macro model
is well suited to deal with the asymptotic model in the diffusion limit. Indeed for small $\varepsilon$, the first equation of
(\ref{mM}) by using (\ref{co}) and (\ref{M}) yields
\begin{equation} \label{IN}
g= \mathcal{T}_0^{-1}(v M \cdot \nabla_{x} n)- \mathcal{T}_0^{-1}(\mathcal{T}_1(S)(M(v)n))+O(\varepsilon).
\end{equation}
\noindent  Inserting (\ref{IN}) into (\ref{er3}) yields the asymptotic model (coupled with the concentration equation for $S$):
\begin{equation}\label{mM2}\resizebox{.9\hsize}{!}{$
\left\{
\begin{array}{l}
    \frac{\partial n}{\partial t} +  \left\langle v. \nabla_x
    (\mathcal{T}_0^{-1}(v.\nabla_x n ))\right\rangle -\left\langle v. \nabla_x
    (\mathcal{T}_0^{-1}(\mathcal{T}_1(S)(M(v)n))\right\rangle =O(\varepsilon),\\
{}  \\ 
   \frac{\partial S}{\partial t}-D_S\Delta S = H(n,S).
\end{array} \right.$}
\end{equation}

Using  iii) of Lemma
 \ref{LE1}, one has as  $\mathcal{T}_0$ is a self adjoint operator in $L^2(\frac{dv}{M(v)})$ the following:
\begin{eqnarray*}\left\langle v. \nabla_x
(\mathcal{T}_0^{-1}(\mathcal{T}_1(S)(M(v)n))\right\rangle
&=&\left\langle \mathcal{T}_0(\theta(v)). \frac{\nabla_x}{M(v)}
(\mathcal{T}_0^{-1}(\mathcal{T}_1(S)(M(v)n))\right\rangle
\nonumber\\  &=& div_x \left\langle
\frac{\theta(v)}{M(v)}n\mathcal{T}_1(S)(M(v))\right \rangle,
\end{eqnarray*}
and consequently the macroscopic model (\ref{mM2}) writes
\begin{equation}\label{mM3}
\left\{
\begin{array}{l}
 \frac{\partial n}{\partial t} + div_{x} \, (n \, \alpha(S)-D_n  \nabla_x n)=O(\varepsilon),  \\
   {}  \\
   \frac{\partial S}{\partial t}-D_S\Delta S = H(n,S),
\end{array} \right.
\end{equation}
where  $ D_n$ and $\alpha(S)$ are given by
\begin{equation}\label{Dn}
 \qquad D_n =- \int_V
v \otimes \theta(v) dv, \quad \alpha(S)= - \int_V  \frac{\theta(v)
}{ M(v)}\mathcal{T}_1(S) (M )(v) dv.
\end{equation}

Our approach appears to be quite general, while the Keller-Segel model can be derived.
Let us  consider  probability kernels such that $T_0(v,v')=\sigma M(v), \quad \sigma >0$.
Consequently, the leading turning operators $\mathcal{T}_0$ become relaxation operators:
\begin{equation}
\mathcal{T}_0(g)= -\sigma \Big(g-\left\langle g\right\rangle
M\Big). \end{equation}
In particular, $\theta$ and  the diffusion tensor $ D_n$ are given by:
\begin{equation}
\theta(v)= -\frac{1}{\sigma}\,   M (v), \quad D_n= \frac{1}{\sigma}\int_V v \otimes v M(v) dv.
\label{DN}\end{equation}
In addition, $\alpha(S)$ is given by:
\begin{equation}
\alpha(S)=  \frac{1}{\sigma} \int_V v \mathcal{T}_1[ S](M(v)) dv.
\label{alpha}
\end{equation}
Then, together with the choice $T_1[S]= K_{S}(v, v')\cdot \nabla_{x} S$,
where $K_{S}(v, v')$ is a vector valued function, yields
$\mathcal{T}_1[S](M)= h(v,S)\cdot \nabla_x S$, where 
$$
h(v,S)=\int_V \Big( K_S(v, v')M(v')  -  K_S( v', v)M(v)\Big)dv'.
$$

Finally,  the function $\alpha(S)$ in \eqref{alpha} is given by $\alpha(S)= \chi(S)\cdot \nabla_x S$,
where the chemotactic sensitivity $\chi(S)$ is given by the matrix
\begin{equation}
\chi(S)=\frac{1}{\sigma} \int_V v\otimes h(v, S) dv. \label{chiS}
\end{equation}

Therefore,  the drift term $div_{x} (n \, \alpha (S))$ that
appears in  the macroscopic case stated by (\ref{mM2}) becomes
$$div_{x} (n \, \alpha (S))= div_{x} \left( n\, \chi (S)  \nabla_{x} {S} \right),$$
which gives a Keller-Segel type model (\ref{ma}):
\begin{equation}\label{K}
\left\{
\begin{array}{l}
    \partial_t n + div_{x} \, \left( n\, \chi (S) \cdot
    \nabla_{x} {S}  -D_n  \nabla_x n \right)=O(\varepsilon),  \\{}\\
   \partial_t S -D_S\Delta S = H(n,S).
\end{array} \right.
\end{equation}

 \section{Numerical methods}

 Let us now  consider Problem (\ref{ki}), subject to the following initial conditions:
    $f(0,x,v) = f_0(x,v)$ and $ S(0,x) = S_0(x)$.
It has been shown that problem (\ref{ki}) is equivalent to the following micro-macro formulation:
 \begin{equation}\label{eq13a}
    \left\{
    \begin{array}{l}
        \frac{\partial g}{\partial t} + \frac{1}{\varepsilon^2}  v M \cdot \nabla_{x} n + \frac{1}{\varepsilon}(I-P_M)(v \cdot
        \nabla_{x} g ) = \frac{1}{\varepsilon^2}\mathcal{T}_0(g) \\
        {}\\
         \qquad\qquad\qquad + \frac{1}{\varepsilon^2}\mathcal{T}_1(S)(M(v)n)
        + \frac{1}{\varepsilon}\mathcal{T}_1(S)(g),  \\
        {}\\
        \frac{\partial n}{\partial t} +  \langle v \cdot \nabla_{x} g \rangle = 0,  \\
        {}\\
        \frac{\partial S}{\partial t}-D_S\Delta S = H(n,S),
    \end{array} \right.
\end{equation}
 subject to the following initial conditions:
 \begin{equation}\label{eq13b}\resizebox{.86\hsize}{!}{$
    n(t = 0) = n_0 = \langle f_0 \rangle,\;
    g(t = 0) = g_0 = \frac{1}{\varepsilon} (f_0-M(v)n_0),\; S(0,x) =
    S_0(x).$}
 \end{equation}

  The discretization of problem (\ref{eq13a})-(\ref{eq13b}) is carried out for each independent variable (time, velocity and space).

 \subsection{Time discretization}

  The treatment of the time variable of problem (\ref{eq13a})-(\ref{eq13b}) can be done by using varieties of methods such as finite difference and variational methods. Finite-differentiation of the derivative in time is the widely used approach.

 The time interval $[0,T]$ is divided into $N$ times steps as follows:  $t_0=0, \quad t_{k+1}=t_k+\Delta t, \quad 0\leq k < N,$
 where $\Delta t = \frac{T}{N}$ is the time step. The approximation of $n(t,x)$ and $g(t,x,v)$ at the time step $t_k$ are denoted respectively by $n^k\approx n(t_k,x)$ and $g^k\approx g(t_k,x,v)$.  Using an implicit scheme for the stiff term $\frac{1}{\varepsilon^2}\mathcal{T}_0(g)$ and an explicit for the other terms in the first equation in (\ref{eq13a}), one obtains :
 \begin{eqnarray}\nonumber
    \frac{g^{k+1}-g^k}{\Delta t}&=&- \frac{1}{\varepsilon^2}  v M \cdot \nabla_{x} n^k - \frac{1}{\varepsilon}(I-P_M)(v \cdot
        \nabla_{x} g^k ) \\ \label{eq14a}
        &+& \frac{1}{\varepsilon^2}\mathcal{T}_0(g^{k+1})+ \frac{1}{\varepsilon^2}\mathcal{T}_1(S^k)(M(v)n^k) +\frac{1}{\varepsilon}\mathcal{T}_1(S^k)(g^k).
 \end{eqnarray}
 Substituting $g$ by $g^{k+1}$ in the second equation of (\ref{eq13a}) yields
 \begin{equation}\label{eq14b}
        \frac{n^{k+1}-n^k}{\Delta t} +  \langle v \cdot \nabla_{x} g^{k+1} \rangle = 0.
 \end{equation}
 Replacing $n$ in the third equation by $n^{k+1}$ one has:
 \begin{equation}\label{eq14c}
        \frac{S^{k+1}-S^k}{\Delta t}-D_{S^k}\Delta S^{k+1} = H(n^{k+1},S^{k+1}).
 \end{equation}
 \begin{Proposition}\label{propN1}
    The time discretization (\ref{eq14a})-(\ref{eq14b}) of  the first and second equation of system (\ref{eq13a}) is consistent
    with the first equation of system (\ref{mM2}) when $\varepsilon \longrightarrow 0$.
 \end{Proposition}

 \begin{proof}
     Formally,  (\ref{eq14a}) yields:
    \begin{eqnarray}\nonumber
        \left(I-\frac{\Delta t}{\varepsilon^2}\mathcal{T}_0\right)g^{k+1}  &=& g^k +
         \frac{\Delta t}{\varepsilon^2} \left[\mathcal{T}_1(S^k)\left(M(v)n^k+\varepsilon g^k\right) - vM\cdot \nabla_x n^k \right.\\ \label{eq15}
                        & & \left.- \varepsilon (I-P_M)v\cdot \nabla_x g^k\right].
    \end{eqnarray}

    Since the operator $-\mathcal{T}_0$ is self adjoint and positive defined, also
    $\left(I-\frac{\Delta t}{\varepsilon^2}\mathcal{T}_0\right)$ is self-adjoint and positive
    definite thus invertible for $\Delta t >0$. Therefore one has
    \begin{eqnarray}
        \nonumber
        g^{k+1}  &=& \left(I-\frac{\Delta t}{\varepsilon^2}\mathcal{T}_0\right)^{-1} \left( g^k +
        \frac{\Delta t}{\varepsilon^2} \left[\mathcal{T}_1(S^k)\left(M(v)n^k+\varepsilon g^k\right) - vM\cdot \nabla_x n^k \right. \right.\\\label{eq16}
        & & \left. \left.- \varepsilon (I-P_M)v\cdot \nabla_x g^k\right] \right).
    \end{eqnarray}
    Developing the right hand side of (\ref{eq16}) with regard to $\varepsilon$ when $\varepsilon \longrightarrow 0$, yields:
    \begin{equation}\label{eq17}
        g^{k+1} = \mathcal{T}_0^{-1}\left[vM\cdot \nabla_x n^k\right] - \mathcal{T}_0^{-1}\left[\mathcal{T}_1(S^k)(M(v)n^k)\right] + O(\varepsilon).
    \end{equation}
    Substituting $g^{k+1}$ into (\ref{eq14b}) leads to
    \begin{equation}\label{eq18}\resizebox{.9\hsize}{!}{$
        \frac{n^{k+1}-n^k}{\Delta t} +
        \langle v \cdot \nabla_{x} \mathcal{T}_0^{-1}\left[vM\cdot \nabla_x n^k\right] \rangle -
         \langle v \cdot \nabla_{x} \mathcal{T}_0^{-1}\left[\mathcal{T}_1(S^k)(M(v)n^k)\right] \rangle= O(\varepsilon),$}
    \end{equation}
    which is consistent with the first equation of system (\ref{mM2}) when $\varepsilon \longrightarrow 0$.
 \end{proof}

\subsection{Space and velocity dicretization}

In the following, we present the methods in the case of 1-dimensional space and velocity discretization. The phase-space
interval is denoted by $[x_{min},x_{max}]\times [v_{min},v_{max}]$, where $-v_{min}=v_{max}>0$.

Focusing on the  \textit{spatial discretization}, a finite difference method based on
control volume approach and cell averaging is used. The numerical grid is defined by:
$$
 \displaystyle{ R_{\Delta x} = \{x_i,x_{i+\frac{1}{2}}, \quad 0\leq i \leq N_x=\frac{x_{max}-x_{min}}{\Delta x} \},}
$$
where $\Delta x >0$ is the spatial mesh size, $x_0 = x_{min}$, $x_i = x_{i-1} + \Delta x$ $(1\leq i \leq N_x)$ and
$x_{i+\frac{1}{2}} = (x_{i+1}+x_i)/2$ $(0\leq i \leq N_x-1)$ are the cell center points. Proceeding as in \cite{Benn09}, the
microscopic equation (\ref{eq14a}) is discretized at points $x_{i+\frac{1}{2}}$ while  the macroscopic equation (\ref{eq14b})
and the diffusion equation  (\ref{eq14c}) are discretized at points $x_i$. The approximation of $n(t,x)$, $g(t,x,v)$ and
$S(t,x)$ at the considered spatial points and at the time step $t_k$ are denoted by $n^k_{i}\approx n(t_k,x_i)$,
$g^k_{i+\frac{1}{2}} (v) \approx g(t_k,x_{i+\frac{1}{2}},v)$ and $S^k_{i}\approx S(t_k,x_i)$ respectively.

Setting $v^+=\max(0,v)$, $v^-=\min(0,v)$, the following spatially discrete forms are obtained:
\begin{eqnarray}\nonumber
    \frac{g^{k+1}_{i+\frac{1}{2}}-g^k_{i+\frac{1}{2}}}{\Delta t}  + \frac{1}{\varepsilon}(I-P_{M})\left(v^+
         \frac{g^k_{i+\frac{1}{2}}-g^k_{i-\frac{1}{2}}}{\Delta x} + v^-
         \frac{g^k_{i+\frac{3}{2}}-g^k_{i+\frac{1}{2}}}{\Delta x}\right)\\
\nonumber
    = \frac{1}{\varepsilon^2}\left(\mathcal{T}_{0}(g^{k+1}_{i+\frac{1}{2}})+ \mathcal{T}_1(S^k_{i+\frac{1}{2}})(M(v)n^k_{i+\frac{1}{2}})
          - v M \cdot \frac{n^k_{i+1}-n^k_i}{\Delta x}\right)\\
 \label{eq19}
    + \frac{1}{\varepsilon}\mathcal{T}_1(S^k_{i+\frac{1}{2}})(g^k_{i+\frac{1}{2}}),\qquad\qquad\qquad\qquad\qquad\qquad\qquad\qquad\;\;\;\;\;\\
 \label{eq20}
        \frac{n^{k+1}_{i}-n^k_i}{\Delta t} +  \left\langle v \frac{g^{k+1}_{i+\frac{1}{2}}-g^{k+1}_{i-\frac{1}{2}}}{\Delta x} \right \rangle =
        0,\qquad\qquad\qquad\qquad\qquad
\end{eqnarray}
 and
 \begin{equation}\label{eq21}
        \frac{S^{k+1}_i-S^k_i}{\Delta t}-D_{S^k_i}\frac{S^{k+1}_{i-1}-2S^{k+1}_i+S^{k+1}_{i+1}}{(\Delta x)^2} = H(n^{k+1}_i,S^{k+1}_i).
 \end{equation}

 \begin{Proposition}\label{propN2}
    From the discretization (\ref{eq19})-(\ref{eq21}) of (\ref{ki}), yields  the following numerical scheme when  $\varepsilon \rightarrow 0$:
    \begin{eqnarray}\nonumber
        \frac{n^{k+1}_{i}-n^k_i}{\Delta t}& +&  \frac{1}{\Delta x}\left\langle v \left[ \mathcal{T}^{-1}_{0}\left( v M \cdot \frac{n^k_{i+1}-n^k_i}{\Delta x}\right) \right.\right.\\ \nonumber
        &-& \left.\left. \mathcal{T}^{-1}_{0}\left( v M \cdot \frac{n^k_{i}-n^k_{i-1}}{\Delta x}\right)\right] \right\rangle \\ \nonumber
        &-&\frac{1}{\Delta x}\left\langle v \left[ \mathcal{T}^{-1}_{0}\left( \mathcal{T}_1(S^k_{i+\frac{1}{2}})(M(v)n^k_{i+\frac{1}{2}})\right)
        \right.\right.\\ \label{eq23a}
        &-& \left.\left. \mathcal{T}^{-1}_{0}\left( \mathcal{T}_1(S^k_{i-1/2})(M(v)n^k_{i-1/2})\right)\right] \right\rangle=0
 \end{eqnarray}
    which is consistent with the first equation of system (\ref{mM2}). Moreover, the approximation of the diffusion term is second order accurate in space.
 \end{Proposition}

 \begin{proof}
 The quantity $g^k_{i+\frac{1}{2}}$ is derived from (\ref{eq19}) as:
 \begin{eqnarray}\nonumber
    g^{k+1}_{i+\frac{1}{2}} &=& \left(I-\frac{\Delta t}{\varepsilon^2}\mathcal{T}_{0}\right)^{-1}
        \left[g^k_{i+\frac{1}{2}}\right.\\ \nonumber
        && - \frac{\Delta t}{\varepsilon} (I-P_{M})\left(v^+
         \frac{g^k_{i+\frac{1}{2}}-g^k_{i-\frac{1}{2}}}{\Delta x} + v^-
         \frac{g^k_{i+\frac{3}{2}}-g^k_{i+\frac{1}{2}}}{\Delta x}\right)\\ \nonumber
         && + \frac{\Delta t}{\varepsilon^2}\bigg(\varepsilon \mathcal{T}_1(S^k_{i+\frac{1}{2}})(g^k_{i+\frac{1}{2}}) +
    \mathcal{T}_1(S^k_{i+\frac{1}{2}})(M(v)n^k_{i+\frac{1}{2}})   \left. - v M \cdot \frac{n^k_{i+1}-n^k_i}{\Delta
    x}\bigg)\right].
 \end{eqnarray}
 It follows that:
 \begin{eqnarray}\label{eq22}
    g^{k+1}_{i+\frac{1}{2}} &=& \mathcal{T}^{-1}_{0}\left[ v M \cdot \frac{n^k_{i+1}-n^k_i}{\Delta x} -
     \mathcal{T}_1(S^k_{i+\frac{1}{2}})(M(v)n^k_{i+\frac{1}{2}}) \right] +O(\varepsilon)
 \end{eqnarray}
as $\varepsilon \longrightarrow 0$. Replacing $g^{k+1}_{i+\frac{1}{2}}$ in (\ref{eq20}) by its expression of (\ref{eq22}) and passing to the limit yields the relation (\ref{eq23a}). Proceeding as in the continuous case, the spatially discrete form (\ref{eq23a}) is  consistent  with the first equation of system (\ref{mM2}).
\end{proof}

Let us now focus on the \textit{velocity discretization} and consider a uniform velocity grid defined as:
$
   \displaystyle{ V_{\Delta v} = \{v_j =v_{min} +j\Delta v,\; 0\leq j \leq N_v \},}
$
where $\Delta v = \frac{v_{max}-v_{min}}{N_v}$ is the velocity step and $N_v\in \mathbb{N}^*$ is and odd number. The
approximation of $g(t,x,v)$  at the spatial points $x_{i+\frac{1}{2}}$ and velocity $v_j$ at  time step $t_k$ is
denoted by $g^k_{i+\frac{1}{2},j} \approx g(t_k,x_{i+\frac{1}{2}},v_j)$.   The velocity discretization is achieved by substituting $v$ by $v_j$ and $g^k_{i+\frac{1}{2}}(v_j)$ by $g^k_{i+\frac{1}{2},j}$ in Equations (\ref{eq19})-(\ref{eq20}), and numerically approximate integrals therein. The bracket $\langle . \rangle$, the projection $P_{M_j}$ and the integral operators $\mathcal{T}_{0,j}$ and $\mathcal{T}_{1,j}(S_{i+\frac{1}{2}}^k)$, are approximated using the trapezoidal rule. 

 The numerical study  of  (\ref{ki}) needs \textit{boundary conditions}.  The following inflow boundary conditions are usually applied to $f$:
 \begin{equation}\label{eq32}
    f(t,x_{min},v) = f_{l}(v),\; v>0 \quad \text{and} \quad   f(t,x_{max},v) = f_{r}(v),\; v<0,
 \end{equation}
  which can be rewritten in the micro-macro formulation:
 \begin{equation}\label{eq33}\resizebox{.9\hsize}{!}{$
    \begin{array}{ll}
        n(t,x_0)M_j+\frac{\varepsilon}{2}\left(g(t,x_{\frac{1}{2}},v_j) + g(t,x_{-\frac{1}{2}},v_j)\right)=f_{l}(v_j),& v_j>0,\\
        n(t,x_{N_x})M_j+\frac{\varepsilon}{2}\left(g(t,x_{N_x+\frac{1}{2}},v_j) + g(t,x_{N_x-\frac{1}{2}},v_j)\right)=f_{r}(v_j),& v_j<0,
    \end{array}$}
 \end{equation}
 while, the following artificial Neumann boundary conditions are imposed for the other velocities \cite{L6}:
 \begin{equation}\label{eq34}
    \begin{array}{ll}
        g(t,x_{-\frac{1}{2}},v_j) = g(t,x_{\frac{1}{2}},v_j),& v_j<0,\\{}\\
        g(t,x_{N_x+\frac{1}{2}},v_j) = g(t,x_{N_x-\frac{1}{2}},v_j),& v_j>0.
    \end{array}
 \end{equation}
 Therefore, the "ghost" points can be computed as follows:
 \begin{equation}\label{eq35}\resizebox{.9\hsize}{!}{$
    \left\{
        \begin{array}{l}
          g^{k+1}_{-\frac{1}{2},j} =\frac{2}{\varepsilon}\left(f_l(v_j)-n_0^{k+1}M_j\right)-g^{k+1}_{\frac{1}{2},j},\;\; g^{k+1}_{N_x+\frac{1}{2},j} =g^{k+1}_{N_x-\frac{1}{2},j}, \;\; v_j>0, \\{}\\
          g^{k+1}_{N_x+\frac{1}{2},j} = \frac{2}{\varepsilon}\left(f_r(v_j)-n_{Nx}^{k+1}M_j\right)-g^{k+1}_{N_x-\frac{1}{2},j},\;\;
            g^{k+1}_{-\frac{1}{2},j} =g^{k+1}_{\frac{1}{2},j}, \;\; v_j<0.
        \end{array}\right.$}
 \end{equation}
 Then it follows from (\ref{eq20}) that:
 \begin{equation}\label{eq36}\resizebox{.9\hsize}{!}{$
    \left\{
    \begin{array}{l}
      \left(1+\frac{2\Delta t}{\varepsilon \Delta x}\langle v^+_j M_j\rangle\right) n^{k+1}_{0} =n^k_0 -
    \frac{\Delta t}{\Delta x}\left\langle (v_j +|v_j|)g^{k+1}_{\frac{1}{2},j}
    -\frac{2v_j^+}{\varepsilon}f_l(v_j)\right\rangle,\\
      \left(1-\frac{2\Delta t}{\varepsilon \Delta x}\langle v^-_j M_j\rangle\right) n^{k+1}_{N_x} =
    n^k_{N_x} - \frac{\Delta t}{\Delta x}\left\langle\frac{2v_j^-}{\varepsilon}f_r(v_j) \right.
     - \left.(v_j-|v_j|) g^{k+1}_{N_x-\frac{1}{2},j}\right\rangle.
    \end{array}\right.$}
 \end{equation}

 In addition, the homogeneous Neumann boundary conditions are prescribed for the concentration $S$:
    $S^{k+1}_{-1}= S^{k+1}_1$  and $S^{k+1}_{N_x+1}= S^{k+1}_{N_x-1}$.

 Focusing now on the \textit{implementation of the method}, the following algorithm can be used for the numerical solution of the Micro-Macro system (\ref{mM}):
 Given $g^0$, $g^0_{-\frac{1}{2}}$, $g^0_{N_x+\frac{1}{2}}$, $S^0$, $n^0_{in}$, $n^0_0$, $n^0_{N_x}$.\\
 For $k=1,2,\cdots,N$
    \begin{enumerate}
        \item Solve $g^{k+1}_{i+\frac{1}{2}}$, ($i=0,1,\cdots, N_x-1$) from (\ref{eq19});
        \item Compute $n^{k+1}_{i}$ ($i=1,2,\cdots, N_x-1$)using (\ref{eq20});
        \item Compute $n^{k+1}_0$ and $n^{k+1}_{N_x}$ using (\ref{eq36});
        \item Compute $g^{k+1}_{-\frac{1}{2}}$ and $g^{k+1}_{N_x+\frac{1}{2}}$ using (\ref{eq35});
        \item Solve $S^{k+1}$ from (\ref{eq21}).
    \end{enumerate}

\subsection {A time implicit discretization}
The previous discretization  is explicit for the macro part $n$. It then imposes the diffusion restriction on the time step $\Delta t= O((\Delta x)^2)$. To overcome this restriction, a time implicit discretization can be applied for the macro part such that at the limit, the diffusion term is treated implicitly. Following the idea in \cite{L6}, a time implicit scheme can be derived for the macro part of the micro-macro system. It consists to substitute $g^{k+1}_{i\pm\frac{1}{2}}$ in (\ref{eq20}) by $\tilde{g}^{k+1}_{i\pm\frac{1}{2}}$ defined as follows:
\begin{eqnarray}\nonumber
    \tilde{g}^{k+1}_{i+\frac{1}{2}} &=&  -\frac{\Delta t}{\varepsilon^2}\mathcal{H}_{0\varepsilon}^{-1}\left(v M \cdot \partial_x n^{k+1}_{i+\frac{1}{2}}\right) +\mathcal{H}_{0\varepsilon}^{-1}
        \bigg[g^k_{i+\frac{1}{2}}\\ \nonumber
        && - \frac{\Delta t}{\varepsilon} (I-P_{M})\left(v^+
         \frac{g^k_{i+\frac{1}{2}}-g^k_{i-\frac{1}{2}}}{\Delta x} + v^-
         \frac{g^k_{i+\frac{3}{2}}-g^k_{i+\frac{1}{2}}}{\Delta x}\right)\\ \label{im1}
         && + \frac{\Delta t}{\varepsilon^2}\left(\varepsilon \mathcal{T}_1(S^k_{i+\frac{1}{2}})(g^k_{i+\frac{1}{2}}) +
    \mathcal{T}_1(S^k_{i+\frac{1}{2}})(M(v)n^k_{i+\frac{1}{2}})\right) \bigg],
\end{eqnarray}
where
$\mathcal{H}_{0\varepsilon}=\left(I-\frac{\Delta t}{\varepsilon^2}\mathcal{T}_{0}\right)$ and $\partial_x n^{k+1}_{i+\frac{1}{2}}=\frac{n^{k+1}_{i+1}-n^{k+1}_i}{\Delta x}.$

Therefore, the following implicit time discretization of the macro part is obtained:
\begin{eqnarray}\label{im2}\resizebox{.95\hsize}{!}{$
        \frac{n^{k+1}_{i}-n^k_i}{\Delta t} -\bigg\langle \frac{\Delta t}{\varepsilon^2}v\mathcal{H}_{0\varepsilon}^{-1}\left(v M \frac{\partial_x n^{k+1}_{i+\frac{1}{2}}-\partial_x n^{k+1}_{i-\frac{1}{2}}}{\Delta x}\right)\bigg\rangle 
        +  \left\langle v \frac{\hat{g}^{k+1}_{i+\frac{1}{2}}-\hat{g}^{k+1}_{i-\frac{1}{2}}}{\Delta x} \right \rangle =
        0,$}
 \end{eqnarray}
 where
 $$\hat{g}^{k+1}_{i+\frac{1}{2}} = {g}^{k+1}_{i+\frac{1}{2}} + \frac{\Delta t}{\varepsilon^2}\mathcal{H}_{0\varepsilon}^{-1}\left(v M \cdot \partial_x n^k_{i+\frac{1}{2}}\right).$$
 It can be seen that for small $\varepsilon$,
 $$\frac{\Delta t}{\varepsilon^2}\mathcal{H}_{0\varepsilon}^{-1}\left(v M \cdot \partial_x n^{k+1}_{i+\frac{1}{2}}\right) = -\mathcal{T}_{0}^{-1}\left(v M \cdot \partial_x n^{k+1}_{i+\frac{1}{2}}\right) + O(\varepsilon)$$
 and
 $$\hat{g}^{k+1}_{i+\frac{1}{2}}=\mathcal{T}_{0}^{-1}\left(\mathcal{T}_1(S^k)(M(v)n^k_{i+\frac{1}{2}})\right).$$
 Trough substitution and using the properties of $\mathcal{T}_{0}$, we obtain the following discret form of the macro part as $\varepsilon \longrightarrow 0$
 \begin{equation}\label{im3}\resizebox{.9\hsize}{!}{$
        \frac{n^{k+1}_{i}-n^k_i}{\Delta t} - D_n\frac{ n^{k+1}_{i+1}-2n^{k+1}_{i}+n^{k+1}_{i+1}}{(\Delta x)^2} 
        +  \frac{\alpha(S_{i+\frac{1}{2}})n^k_{i+\frac{1}{2}}-\alpha(S_{i-\frac{1}{2}})n^k_{i-\frac{1}{2}}}{\Delta x} =
        0,$}
 \end{equation}
which is consistent with a discrete form  of the macroscopic limit, obtained by using an implicit dicretization of the diffusion term. We remark that there is an additional computation of  $\mathcal{H}_{0\varepsilon}^{-1}$ for the calculation of $n^{k+1}$. In the particular case, where $$\mathcal{T}_{0}(f)=-\sigma( f-\langle f \rangle M(v)),$$ we have from the micro-macro decomposition $\mathcal{T}_{0}({g}^{k+1}_{i+\frac{1}{2}})=-\sigma {g}^{k+1}_{i+\frac{1}{2}}$. Hence

\begin{equation}\label{im4}\resizebox{.99\hsize}{!}{$
        \frac{n^{k+1}_{i}-n^k_i}{\Delta t} - \frac{\Delta t}{\varepsilon^2+\sigma \Delta t}\bigg\langle \left(v^2 M \frac{\partial_x n^{k+1}_{i+\frac{1}{2}}-\partial_x n^{k+1}_{i-\frac{1}{2}}}{\Delta x}\right)\bigg\rangle
        +  \left\langle v \frac{\hat{g}^{k+1}_{i+\frac{1}{2}}-\hat{g}^{k+1}_{i-\frac{1}{2}}}{\Delta x} \right \rangle =
        0,$}
 \end{equation}
 while $n^{k+1}$ is obtained by solving the linear system
 \begin{equation}
    (An^{k+1})_i = n^k_i - \Delta t\left\langle v \frac{\hat{g}^{k+1}_{i+\frac{1}{2}}-\hat{g}^{k+1}_{i-\frac{1}{2}}}{\Delta x} \right \rangle,
 \end{equation}
where $A$ is the tridiagonal matrix $A = Tridiag(-\gamma,2+\gamma,-\gamma)$ with
$$\gamma = \langle v^2M\rangle\frac{(\Delta x)^2}{(\Delta t)^2(\varepsilon^2+\Delta t)}\quad \text{and} \quad
\hat{g}^{k+1}_{i+\frac{1}{2}} = g^{k+1}_{i+\frac{1}{2}}+\frac{\Delta t}{\varepsilon^2+\sigma \Delta t}\left(v M \partial_x n^{k}_{i+\frac{1}{2}}\right).$$
The boundary conditions are incorporated using (\ref{eq36}).

\section{Numerical results}

We present, in this section, some numerical experiments to validate our
approach. In our tests, the space domain is the interval
 $X= [-1; 1]$, the velocity domain is $V = [ -1; 1]$, while for all  numerical tests, the velocity space is divided into $N_v=64$, which can provide good enough accuracy for numerical simulations \cite{Carr11}.  The equilibrium distribution $M(v)$ and the kernels $T_0(v,v')$ and $T_1(S,v,v')$ are chosen as follows:
  $$\displaystyle{ M(v)=\frac{1}{2}}, \quad \displaystyle{ T_0(v,v') = M(v)},\quad \displaystyle{T_1(S,v,v')=(v.\nabla S)_+},$$
 so that $T_0$ and $T_1$ satisfy assumptions (\ref{co})-(\ref{T}).

 Boundary conditions are given by 
 $$f(t,-1,v) = 0, \forall v<0; f(t,1,v) = 0, \forall v > 0.$$
  For the chemoattractant equation, we consider $H(n,S) = -S+n$ and the initial condition $S(0,x)=0$. We consider the following non-equilibrium initial cell distribution function :
$$f(0,x,v) =  80\exp(-80x^2)\exp(v/100)M(v).$$

Our scheme is compared with:
\begin{itemize}
    \item an explicit-Euler scheme applied to the kinetic equation in the kinetic regime;
    \item an explicit finite difference method for the corresponding Keller-Segel system equation in the macroscopic regime \cite{Saito09};
    \item an asymptotic preserving scheme obtained from a time splitting method applied to the Odd-Even decomposition of the kinetic equation \cite{Carr11,JP}, in both kinetic and macroscopic regime.
\end{itemize}

\noindent \vskip.2cm {\bf Numerical tests:}  In the following, we denote by:
\begin{itemize}
    \item MM: the scheme obtained from the micro-macro decomposition,

    \item K-S: the scheme for the keller-Segel system,

    \item Explicit: the explicit scheme for the kinetic equation,

    \item Odd-Even: the odd-even parity asymptotic preserving scheme.
\end{itemize}
We have observed that the use of the time implicit discretization for the micro-macro model and the Keller-Segel limit give rise to numerical results which are very close to those produced by the explicit discretization. Hence, for the numerical results presented, we use the implicit approach for the MM and K-S schemes. The first test concerns the convergence order of the MM scheme computed at time $t$ using the $l^2$ as follows:
$$\displaystyle{e_{\Delta x}(f) = \frac{\|f_{\Delta x}(t)-f_{2\Delta x}(t)\|_2}{\|f_{{2\Delta x}}(0)\|_2},}$$
where $f_{\Delta x}$ denotes the approximation of $f$ using the spatial grid size $\Delta x$. The time step is set to $\Delta t = \frac{(\Delta x)^2}{2}$. Figure~\ref{fig1a} presents the convergence rates obtained with $N_x=80,\,160,\,320, \, 640$ at time $t=0.1$ for $\varepsilon \in \{1,\,0.01,\,10^{-4}, \,10^{-6}\}$. It can be seen  that the MM scheme converges uniformly since time step does not depend on $\varepsilon$. A second order convergence is observed in the diffusive regime ($\varepsilon\leq 10^{-4}$).
In the following, we set $N_x=200$. The time step is set to $\Delta t = \frac{\varepsilon\Delta x}{2}$ at the kinetic regime and $\Delta t = O(\Delta x)$ at the diffusive regime ($\Delta t=\Delta x/2$ for MM and K-S schemes and $\Delta t=\Delta x/40$ for the Odd-Even method). We illustrate in Figure~\ref{fig2} the behaviour of the MM scheme at different regimes. For different values of $\varepsilon$
($\varepsilon_k=2^{-k}$, $k\geq 0$), we plot at time $t=0.5$ the density of cells. We also add the result obtained with the K-S scheme. It can be seen that the MM scheme is stable as $\varepsilon \rightarrow 0$ and converges to the Keller-Segel limit. Indeed, for $\varepsilon \leq 2^{-7}$, the profiles of the density given by the two
schemes are quite the same.

To check the behaviour of MM scheme in kinetic regime, we compare in Figure~\ref{fig3a} the density of cells obtained for $\varepsilon=1$ with the MM, Explicit  and Odd-Even schemes at time $t=0.5$. As expected, for both schemes, the results are very closed.
\begin{figure}[htp]
 \centering
  \includegraphics[width=0.8\textwidth]{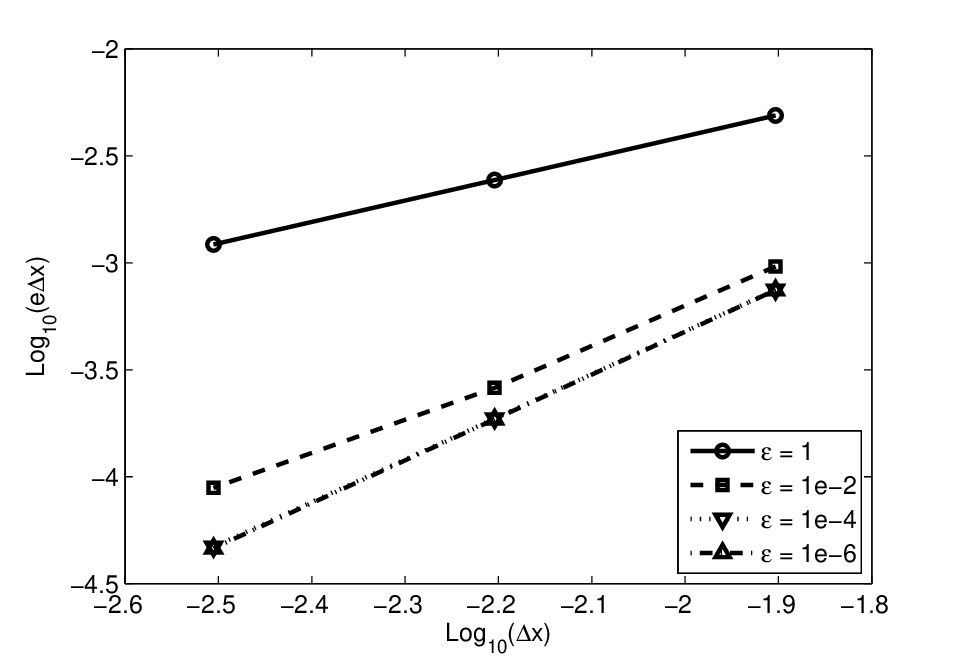}
 \caption[Optional caption for list of figures]{%
  Convergence order of the method for $\varepsilon \in \{1,\,0.01,\,10^{-4}, \,10^{-6}\}$ at time $t=0.1$.}
  \label{fig1a}
\end{figure}
\begin{figure}[htp]
 \centering
 \includegraphics[width=0.8\textwidth]{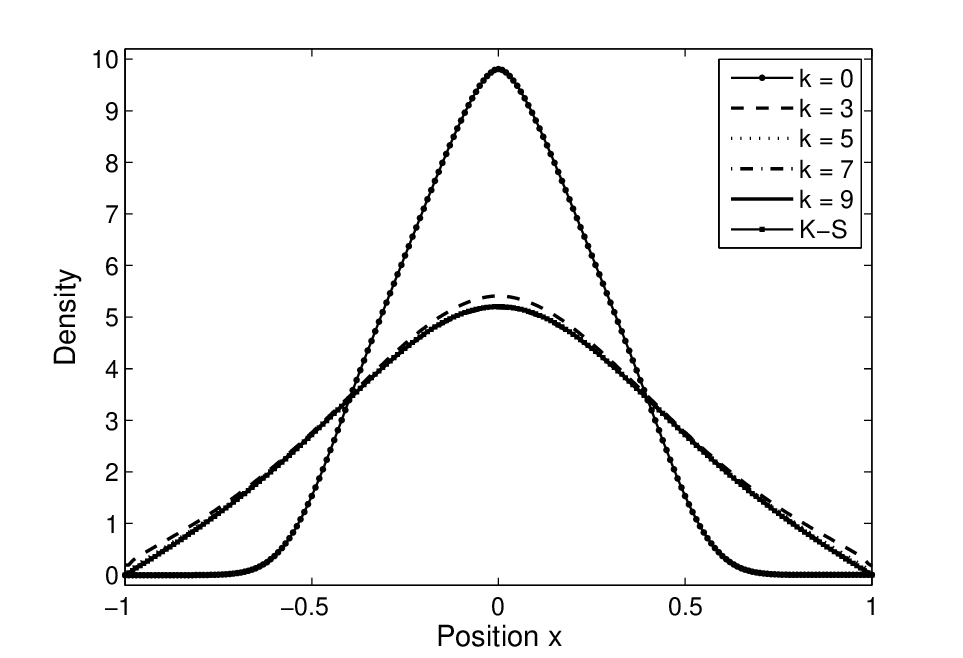}
 \caption[Optional caption for list of figures]{%
  Density of cells at time $t=0.5$ using MM and K-S schemes for $\varepsilon = 2^{-k}$ ($k\in\{0,1,2,5,9\}$).}
  \label{fig2}
\end{figure}
\begin{figure}[htp]
\centering
   \includegraphics[width=0.8\textwidth]{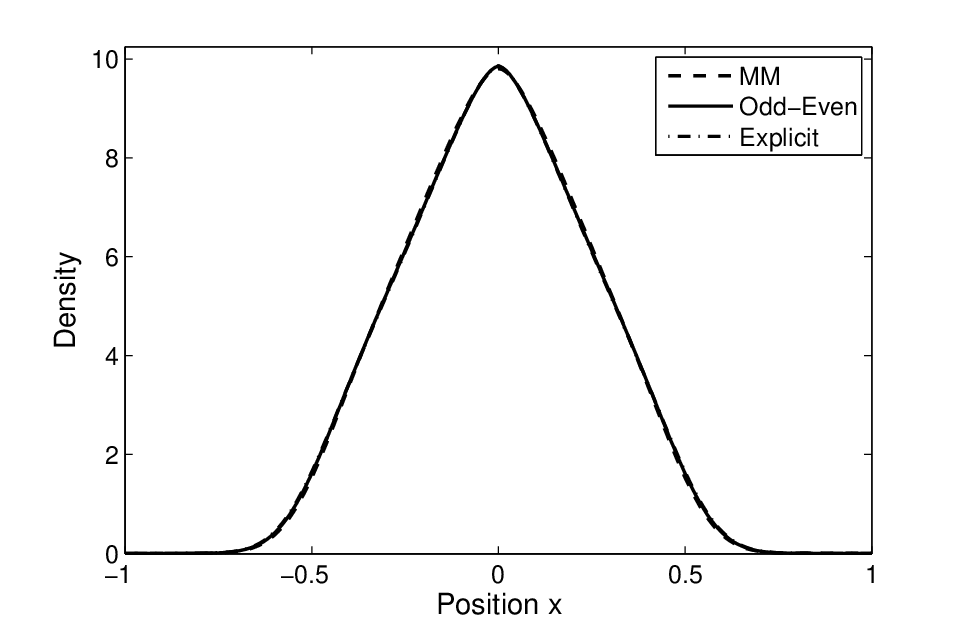}
\caption[0pt]{%
  Density of cells at time $t=0.5$ obtained with MM, Odd-Even and Explicit schemes for $\varepsilon = 1$.}
  \label{fig3a}
\end{figure}
\begin{figure}[htp]
   \centering
  \includegraphics[width=0.8\textwidth]{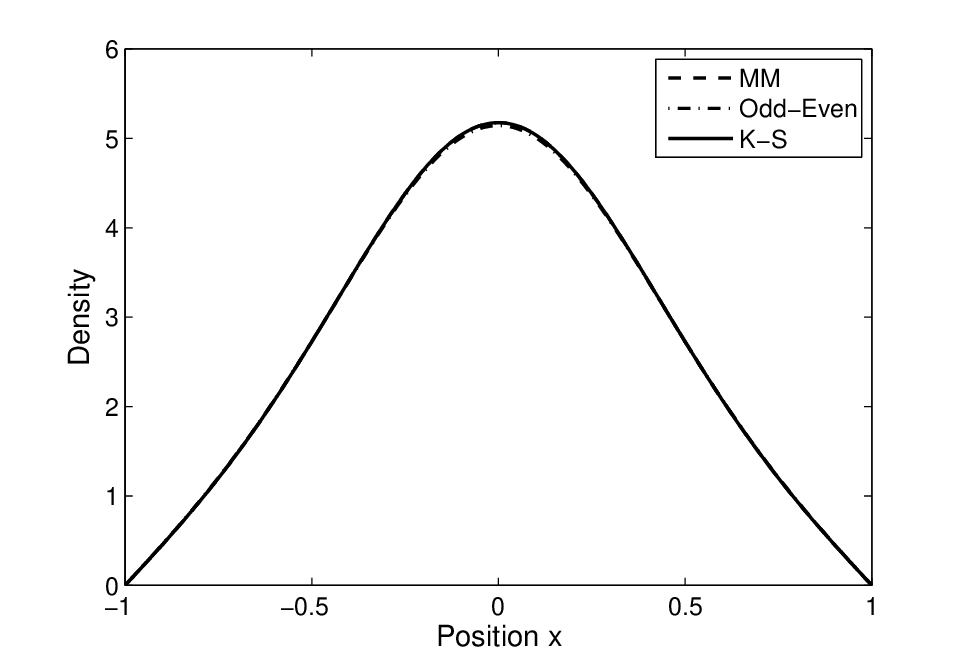}
\caption[0pt]{%
  Density of cells at time $t=0.5$ obtained with MM, Odd-Even and Explicit schemes for $\varepsilon = 10^{-6}$.}
  \label{fig3b}
\end{figure}
We also illustrate the behaviour of the methods in macroscopic regime. We compare in Figure~\ref{fig3b} the density of cells obtained for $\varepsilon=10^{-6}$ with the MM, Odd-Even and K-S schemes at time $t=0.5$. As expected, for both schemes, the results are quite the same.

We investigate the time evolution of the density using the MM scheme in different regimes ($\varepsilon = 1,  10^{-6}$). The
results are shown in Figure~\ref{fig5a} and Figure~\ref{fig5b}. In each case, the density seems to evolve to a stationary solution.
\begin{figure}[htp]
\centering
   \includegraphics[width=0.8\textwidth]{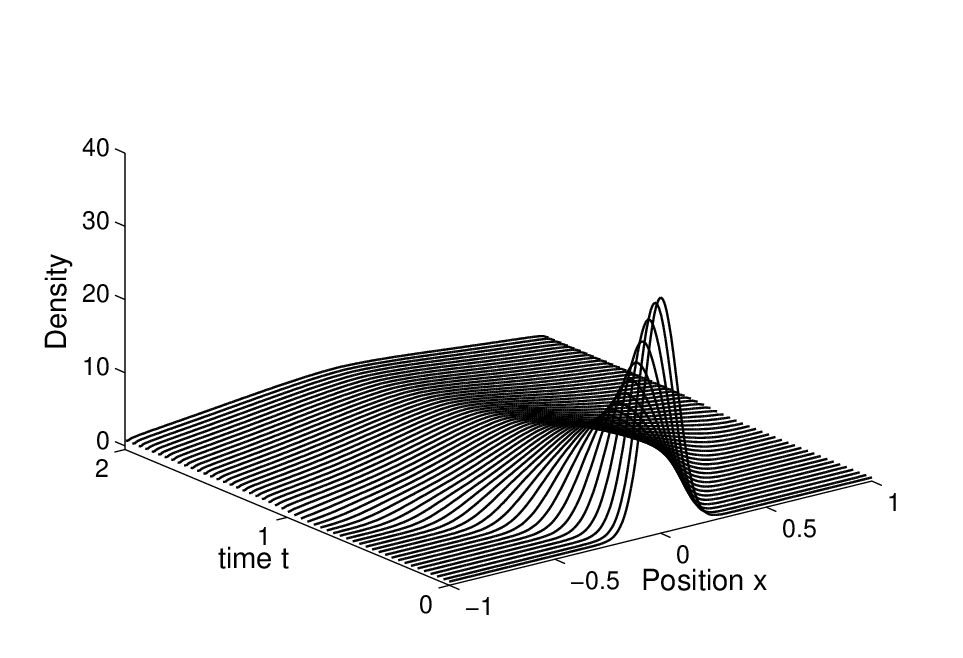}\\ 
   \caption[Optional caption for list of figures]{%
  Evolution of the cell Density using MM scheme for $\varepsilon=1$.}
  \label{fig5a}
\end{figure}
\begin{figure}[htp]
 \centering
   \includegraphics[width=0.8\textwidth]{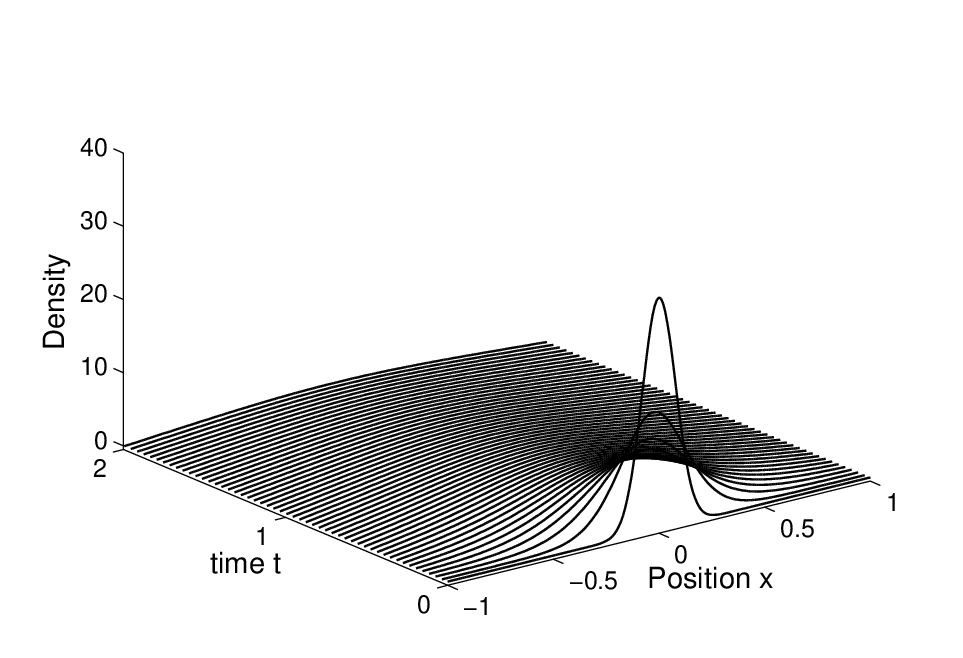}\\ 
\caption[Optional caption for list of figures]{%
  Evolution of the cell Density using MM scheme for $\varepsilon = 10^{-6}$.}
  \label{fig5b}
\end{figure}

\section{Closure looking ahead at research perspectives}

This paper has developed a computational approach to a class of pattern formation models derived from the celebrated Keller-Segel model obtained by the
underlying description delivered by generalized kinetic theory methods.  The derivation is based on a decomposition with two scales, namely the microscopic and
the macroscopic one technically related, as we have seen, by suitable small parameters accounting for the time and space dynamics.

The novelty of our paper is that the computational scheme which follows precisely the derivation hallmarks by using the same decomposition and parameters. This idea
improves the stability properties of the solutions with respect to classical approaches known in the literature. However, without repeating concepts already mentioned in the previous sections, we wish to stress that this method can contribute to future developments also related to applications. In fact, the need of
new models in biology is  presented in \cite{[HP1]} and \cite{[BB5]} to account for a broad variety of biological phenomena. Moreover, it is shown in \cite{[BB5]} that the so-called micro-macro decomposition can lead to an interesting variety of models such as models of angiogenesis phenomena.

Therefore, modeling and computational methods can march together thus contribution to a deeper understanding of the specific features of the two different, however, related fields. Indeed, we have in mind not only applications in biology, but also to the dynamics of self-propelled particles such as those of vehicular traffic as it has been recently shown \cite{[BB4]} how macroscopic models can be derived from the kinetic description using precisely the micro-macro decomposition treated in this present paper.

\section*{Acknowledgments}

 The first author was supported  by Hassan II Academy of Sciences  and Technology (Morocco), project ``M\'ethodes math\'ematiques et
    outils de mod\'elisation et simulation pour le cancer''. Part of this work was done during the visit of the second author
at ENSA Marrakech (Morocco).

\end{document}